\documentclass[oneside, 11pt]{article}

\usepackage[utf8]{inputenc}
\usepackage[T1]{fontenc}
\usepackage{hyperref}
\usepackage{url}
\usepackage{nicefrac}
\usepackage{microtype}
\usepackage{xcolor}
\usepackage{amsfonts, amsmath, mathtools, amssymb, amsthm}
\usepackage{graphicx}
\usepackage{booktabs}
\usepackage{multirow}
\usepackage[numbers]{natbib}
\usepackage[
    letterpaper,
    left=0.95in,
    right=0.95in,
    top=1in,
    bottom=1in
]{geometry}
\usepackage{authblk}
\usepackage{mathrsfs}
\usepackage{textcomp}
\usepackage{manyfoot}
\usepackage{algorithm}
\usepackage{algorithmicx}
\usepackage{algpseudocode}
\usepackage{listings}

\theoremstyle{plain}
\newtheorem{theorem}{Theorem}

\newtheorem{lemma}{Lemma}

\theoremstyle{remark}

\theoremstyle{definition}
\newtheorem{definition}{Definition}
\newtheorem{assumption}{Assumption}

\newcommand{\bR}{\mathbb{R}}
\newcommand{\bE}{\mathbb{E}}
\newcommand{\cX}{\mathcal{X}}
\newcommand{\cC}{\mathcal{C}}
\newcommand{\cO}{\mathcal{O}}
\newcommand{\cN}{\mathcal{N}}

\newcommand{\argmin}[1]{\underset{#1}{\operatorname{arg}\operatorname{min}}\;}

\DeclarePairedDelimiter\abs{\lvert}{\rvert}
\DeclarePairedDelimiter\norm{\lVert}{\rVert}

\date{}
\title{Primal Methods for Variational Inequality Problems \\ with Functional Constraints}

\author[1,2]{Liang Zhang}
\author[1]{Niao He}
\author[2]{Michael Muehlebach}
\affil[1]{Department of Computer Science, ETH Zurich}
\affil[2]{Max Planck Institute for Intelligent Systems}
\affil[ ]{\texttt{\{liang.zhang, niao.he\}@inf.ethz.ch, michael.muehlebach@tuebingen.mpg.de}}

\begin{document}

\maketitle

\begin{abstract}
    Variational inequality problems are recognized for their broad applications across various fields including machine learning and operations research.
    First-order methods have emerged as the standard approach for solving these problems due to their simplicity and scalability. However, they typically rely on projection or linear minimization oracles to navigate the feasible set, which becomes computationally expensive in practical scenarios featuring multiple functional constraints.
    Existing efforts to tackle such functional constrained variational inequality problems have centered on primal-dual algorithms grounded in the Lagrangian function.
    These algorithms along with their theoretical analysis often require the existence and prior knowledge of the optimal Lagrange multipliers.
    In this work, we propose a simple primal method, termed Constrained Gradient Method (CGM), for addressing functional constrained variational inequality problems, without requiring any information on the optimal Lagrange multipliers. We establish a non-asymptotic convergence analysis of the algorithm for Minty variational inequality problems with monotone operators under smooth constraints. 
    Remarkably, our algorithms match the complexity of projection-based methods in terms of operator queries for both monotone and strongly monotone settings, while using significantly cheaper oracles based on quadratic programming. 
    Furthermore, we provide several numerical examples to evaluate the efficacy of our algorithms.
\end{abstract}

\section{Introduction}

Variational inequality problems \citep{stampacchia1964formes, lions1967variational, mancino1972convex} provide a unified framework for modeling optimization and equilibrium seeking problems and have been extensively studied across various disciplines.
Important examples include minimax optimization in machine learning \citep{goodfellow2014generative, gidel2018a, madry2018towards}, Nash equilibrium problems in game theory and economics \citep{nash1950equilibrium, nash1951non, nagurney1998network}, and frictional contact problems in physics \citep{johnson1987contact, christensen1998formulation, sofonea2009variational}.
For an in-depth presentation of the historical development and broader applications of variational inequality problems, we recommend the comprehensive books by \citet{kinderlehrer2000introduction}, and \citet{facchinei2003finite}.

In variational inequality problems, the objective is to find $x^*\in\cC$ such that
\begin{equation}
    F(x^*)^\top (x^* - x) \leq 0, \quad \forall\, x\in\cC.
    \label{eq:vip}
\end{equation}
Here, the constrained set $\cC\subseteq\bR^d$ is compact and convex, and the operator $F:\cC\rightarrow\bR^d$ is continuous, which ensures a nonempty and compact solution set \citep{facchinei2003finite}.
The definition above is also known as the Stampacchia variational inequality problem, and the solution $x^*$ is referred to as a strong solution to the variational inequality problem corresponding to $F$ and $\cC$. Our aim is to find a weak solution $x^*\in\cC$ to the variational inequality problem, where it solves the Minty variational inequality problem such that $F(x)^\top (x^* - x) \leq 0, \forall\, x\in\cC$.
In this work, we also assume the operator $F$ to be monotone, a standard assumption in the literature that encompasses important applications including convex minimization and convex-concave minimax optimization problems \citep{bauschke2011convex, rockafellar1970monotone, rockafellar1976monotone, nemirovski2004prox}.
More applications on monotone variational inequality problems can be found in \citet{thekumparampil2022lifted} and Chapter 1.4 of \citet{facchinei2003finite}. When $F$ is continuous and monotone, strong and weak solution sets are equivalent \citep{minty1962monotone, kinderlehrer2000introduction, diakonikolas2020halpern}, and weak solutions are commonly adopted in the related literature for monotone variational inequality problems \citep{nemirovski2004prox, nesterov2007dual, juditsky2011solving}.

There has been a growing interest in the development of first-order methods for solving monotone variational inequality problems, e.g., the projected gradient method \citep{facchinei2003finite}, the proximal point method \citep{rockafellar1976monotone}, the extragradient method \citep{korpelevich1976extragradient, tseng1995linear, nemirovski2004prox, censor2011subgradient}, the optimistic gradient method \citep{popov1980modification, mokhtari2020convergence, mokhtari2020unified, gorbunov2022last}, the dual extrapolation method \citep{nesterov2007dual, nesterov2006solving}, and the projected reflected gradient method \citep{malitsky2015projected}.
These methods, noted for their simplicity and scalability, involve querying the operator $F$ at specific points and accessing a projection oracle or a linear minimization oracle \citep{jaggi2013revisiting, gidel2017frank} for the feasible set $\cC$. Such oracles are easy to compute for simple feasible sets such as Euclidean balls for the projection oracle or polytopes for the linear minimization oracle. However, for general constrained sets, computing the projection or linear minimization oracle requires solving constrained optimization problems over the feasible set $\cC$. The latter remains challenging even with quadratic or linear objectives.

In this work, we focus on the general functional constrained setting, where the feasible set $\cC$ is described by $m$ convex inequality constraints,  that is,
\begin{equation}
    \cC = \{x\in\bR^d \,|\, g_i(x)\leq 0, \; \forall\, 1\leq i\leq m\}.
    \label{eq:constraint}
\end{equation}
Note that the feasible set $\cC$ does not necessarily allow for projection or linear minimization oracles. 
This encompasses important applications in machine learning including reinforcement learning with safety constraints \citep{xu2021crpo}, constrained Markov potential games \citep{alatur2023provably, jordan2024independent}, generalized Nash equilibrium problems with jointly-convex constraints \citep{facchinei2010generalized, jordan2023first}, and learning with fairness constraints \citep{zafar2019fairness, lowy2022stochastic}.

Previous works have predominantly focused on  primal-dual algorithms based on the (augmented) Lagrangian function to handle the constraints. These algorithms and their convergence guarantees crucially depend on information about the optimal Lagrange multipliers. 
\citet{yang2023solving} proposed an ADMM-based interior point method, later refined by \citet{chavdarova2024a}, and they assumed that either $F$ is strictly monotone or one of $g_i(x)$ is strictly convex to ensure the existence of the central path and boudedness of the optimal Lagrange multiplier.
\citet{boob2023first} extended the constraint extrapolation method \citep{boob2023stochastic} for constrained minimization problems to functional constrained variational inequality problems.
Their guarantees rely on the existence and boundedness of the optimal Lagrange multipliers. Meanwhile, the magnitude of the multipliers is essential to determine the stepsize and affects the convergence rate.

Primal methods serve as an alternative approach to avoid such information on the optimal Lagrange multipliers. They are also simpler to analyze and more straightforward to implement in practice. In fact, primal first-order methods have been extensively studied for functional constrained minimization problems. Prominent examples include Polyak's switching gradient method \citep{polyak1967general, lan2020algorithms, huang2023oracle}, cutting-plane methods \citep{kelley1960cutting}, the level bundle method \citep{lemarechal1995new}, and the constrained gradient descent method \citep{muehlebach2022constraints} motivated from nonsmooth dynamical systems, to just name a few.  However, primal methods have been much less explored for functional constrained variational inequality problems.

\textbf{Our contributions.} \; In this work, we propose a primal method, termed constrained gradient method (CGM; see Algorithm \ref{algo:cgd-vip}), for solving functional constrained variational inequality problems. Our algorithm takes inspiration from the constrained gradient descent method \citep{muehlebach2022constraints} and extends its principles to tackle the more challenging variational inequality problems. Unlike traditional projection-based methods that project each iterate onto the feasible set to handle constraints, CGM instead projects the update direction (velocity) onto a local, sparse, and linear approximation of the feasible set. The latter only requires solving a simple quadratic program with linear constraints. 

We establish the global convergence analysis of CGM under two settings: $(i)$ when the operator $F$ is monotone, and $(ii)$ when $F$ is strongly-monotone.  
Notably, in both settings, CGM enjoys (nearly) the same complexity on querying the operator $F$ as the optimal complexity achieved by projection-based methods, even though the projection oracle is replaced by quadratic programming solvers. For example, when the operator $F(x)$ is monotone,  we show that CGM achieves a weak $\epsilon$-solution with $\cO(1/\epsilon^2)$ queries to $F(x)$ and $\cO(1/\epsilon^2)$ calls to a quadratic programming solver. To the best of our knowledge, our algorithm is the first primal method that achieves the optimal complexity on queries to $F$ for functional constrained monotone variational inequality problems without requiring any information on the optimal Lagrange multipliers. This information includes whether the optimal Lagrange multiplier exists and how its magnitude can be estimated.
A comparison with existing results can be found in Table \ref{tab:compare}.

We further illustrate that the quadratic program at each iteration of CGM allows for efficient implementation and permits even closed-form solutions in special cases such as simplex constraints (Algorithm \ref{algo:cgm-sp}) or when there is only one (nonlinear) constraint function (Algorithm \ref{algo:cgm-1}).  This yields a direct implementation of our algorithm comparable to unconstrained methods.
Empirically, we evaluate the convergence of CGM through several numerical experiments and demonstrate its effectiveness.   

\begin{table}[t]
    \centering
    \caption{Complexity on queries of the operator $F$ to guarantee $\max_{x\in\cC} F(x)^\top (\hat x - x)\leq\epsilon$ (optimality) and $\max_i g_i(\hat x)\leq\epsilon$ (feasibility) for the monotone setting, where $\hat x$ is the output of the algorithm. ACVI \citep{yang2023solving, chavdarova2024a} guarantees last-iterate convergence for strong $\epsilon$-solutions $\max_{x\in\cC} F(\hat x)^\top (\hat x - x)\leq\epsilon$, but it also requires stronger assumptions that either $F$ is strictly monotone or one of $g_i$ is strictly convex. In the table, ``P'' denotes primal methods, while ``PD'' stands for primal-dual methods. Primal-dual methods require the existence of optimal Lagrange multipliers $\lambda^*$, while CGM does not. At each iteration of the algorithm, projected gradient method (PGM) requires access to a projection oracle (PO), Frank-Wolfe method requires access to a linear minimization oracle (LMO), ACVI requires solving a strongly-convex optimization (SCO) sub-problem, and CGM requires solving a simple and possibly sparse quadratic program (QP) $\min_{v\in V_\alpha(x_t)} (1/2)\norm{v+F(x_t)}^2$ (see details in Algorithm \ref{algo:cgd-vip}).}
    \vskip 0.1in
    \begin{tabular}{cccccc}
        \toprule
        Method & Type & Optimality & Feasibility & Per-Iter. Cost & Requirement \\
        \midrule
        PGM \citep{facchinei2003finite} & P & $\cO(1/\epsilon^2)$ & - & PO & - \\
        Frank-Wolfe \citep{thekumparampil2020projection} & P  & $\cO(1/\epsilon^2)$ & - & LMO & - \\
        ACVI \citep{yang2023solving, chavdarova2024a} & PD & $\cO(1/\epsilon^2)$ & $\cO(1/\epsilon^2)$  & SCO & Existence of $\lambda^*$ \\
        ConEx \citep{boob2023first} & PD & $\cO(1/\epsilon^2)$ & $\cO(1/\epsilon^2)$ & - & Existence of $\lambda^*$ \\
        CGM (Algorithm \ref{algo:cgd-vip}) & P & $\cO(1/\epsilon^2)$ & $\cO(1/\epsilon^2)$ & QP & - \\
        \bottomrule
    \end{tabular}
    \label{tab:compare}
\end{table}

\textbf{Literature review.} \; Early advances on iterative methods for solving variational inequality problems can be found in \citet{pang1982iterative, facchinei2003finite, noor2004some}. Here, we focus on the literature on first-order methods. For monotone variational inequality problems, the projected gradient method achieves the optimal convergence rate $\cO(1/\epsilon^2)$ \citep{facchinei2003finite}. If the operator is also Lipschitz, the optimal convergence rate $\cO(1/\epsilon)$ has been established for various algorithms, such as the extragradient \citep{nemirovski2004prox}, optimistic gradient \citep{mokhtari2020convergence}, and the dual extrapolation method \citep{nesterov2007dual}. These results focused on the deterministic setting with guarantees on the average iterates. Recent efforts include extension to the stochastic setting \citep{jiang2008stochastic, juditsky2011solving, koshal2012regularized, yousefian2017smoothing, kannan2019optimal, hsieh2019convergence, kotsalis2022simple, huang2022new, alacaoglu2022stochastic} and examination of last-iterate convergence guarantees \citep{golowich2020last, cai2022tight, cai2022finite, gorbunov2022extragradient, gorbunov2022last}. Convergence to strong $\epsilon$-solutions was considered in \citet{dang2015convergence, diakonikolas2020halpern}. Adaptive and parameter-free algorithms were developed in \citet{malitsky2020golden, diakonikolas2020halpern}. There is also a line of research on the relaxation of the monotone assumption, including the quasi-monotone setting \citep{aussel2004quasimonotone, nesterov2006solving, alakoya2022strong}, the pseudo-monotone setting \citep{solodov1999new, kannan2019optimal, boct2020forward}, and the non-monotone setting \citep{song2020optimistic, yang2020global, lee2021fast, diakonikolas2021efficient, pethick2023solving}.
All theses works either consider the unconstrained case or assume access to a projection oracle.

For functional constrained problems, previous works \citep{yang2023solving, chavdarova2024a, boob2023first} centered around primal-dual methods based on the Lagrangian function. Their analysis requires the existence and boundedness of the optimal Lagrange multipliers. Instead, we provide the first primal algorithm, CGM (Algorithm \ref{algo:cgd-vip}), that does not require any information on the optimal Lagrange multipliers. CGM reduces to constrained gradient descent \citep{muehlebach2022constraints} for minimization problems when $F$ is a gradient field. \citet{muehlebach2022constraints} only considered smooth and strongly-convex minimization problems, which form a subclass of Lipschitz and strongly-monotone variational inequality problems. Moreover, they use a stepsize dependent on the optimal Lagrange multipliers and require the quadratic program to be exactly solved. In contrast, our analysis of CGM applies to monontone variational inequality problems without the Lipschitz assumption on the operator $F$ and allows certain inaccuracies in the quadratic programming solver. More importantly, no information on the optimal Lagrange multipliers is required.

\textbf{Organization of the paper.} \; Our paper is organized as follows. The notation and background related to our main results are summarized in Section \ref{sec:pre}. In Section \ref{sec:cgm}, we present CGM and provide its convergence analysis for both monotone and strongly-monotone settings in Section \ref{sec:theory}. In Section \ref{sec:qp}, we discuss different ways to solve the quadratic programs in CGM and give several examples where CGM admits closed-form and direct updates. Numerical results are provided in Section \ref{sec:exp}. Section \ref{sec:conclude} concludes the paper and discusses possible future directions.

\section{Preliminaries} \label{sec:pre}

We rely on the following notation throughout the article.
The Euclidean norm is denoted by $\norm{\cdot}$, and we use $[m]$ to denote the set $\{1,2,\cdots,m\}$.
An operator $F: \cC\rightarrow\bR^d$ defined on a convex set $\cC\subseteq\bR^d$ is $\mu$--strongly-monotone if $\forall x,y\in\cC$, $(F(x) - F(y))^\top (x-y) \geq \mu\norm{x - y}^2$ with $\mu>0$ and monotone if $\mu=0$.
The operator $F$ is $\ell$-Lipschitz if $\norm{F(x) - F(y)} \leq \ell\norm{x - y}$, $\forall x,y\in\cC$.
A function $g: \cX\rightarrow\bR$ defined on a convex set $\cX\subseteq\bR^d$ is convex if $g(\alpha x + (1-\alpha) y)\leq \alpha g(x) + (1-\alpha) g(y)$, $\forall \alpha\in[0,1], \forall x,y\in\cX$.
The function $g(x)$ is $L$-Lipschitz if $\forall x,y\in\cX$, $\abs{g(x) - g(y)}\leq L\norm{x-y}$, and equivalently $\norm{\nabla g(x)} \leq L, \forall x\in\cX$ if it is differentiable.
The function $g(x)$ is $\ell$-smooth if it is differentiable and $\forall x,y\in\cX$, $\norm{\nabla g(x) - \nabla g(y)} \leq \ell\norm{x - y}$.

\subsection{Variational Inequality Problems}

In our paper, we assume that $\cC$ is contained in an Euclidean ball with diameter $D$.
This entails that the convex feasible set $\cC$ is also compact.
By Corollary 2.2.5 in \citet{facchinei2003finite}, the solution set of the variational inequality problem \eqref{eq:vip} is nonempty and compact. We are interested in finding the following weak $\epsilon$-approximate solution of the constrained variational inequality problem \eqref{eq:vip}.

\begin{definition}
    For some $\epsilon>0$, we call a point $\hat x\in\bR^d$ a weak $\epsilon$-solution of the variational inequality problem \eqref{eq:vip}, where the feasible set is of the form \eqref{eq:constraint}, if $F(x)^\top (\hat x - x)\leq\epsilon$, $\forall x\in\cC$ and $g_i(\hat x)\leq\epsilon$, $\forall i\in[m]$.
    \label{def:opt}
\end{definition}

The above definition of the weak $\epsilon$-solution for a variational inequality problem is commonly adopted in the related literature \citep{nemirovski2004prox, nesterov2007dual, juditsky2011solving, boob2023first}.
A strong $\epsilon$-solution, $\hat x \in\cC$, is characterized by $F(\hat x)^\top (\hat x - x) \leq \epsilon$, $\forall x\in\cC$. When $\epsilon=0$ and $F$ is monotone and continuous, the two solution sets are equivalent \citep{minty1962monotone, kinderlehrer2000introduction}. When $\epsilon>0$ and $F$ is monotone, every strong $\epsilon$-solution is also a weak $\epsilon$-solution.

\subsection{Constrained Gradient Descent}

Our main algorithm draws inspiration from the constrained gradient descent (CGD) method, introduced in \citet{muehlebach2022constraints} for solving minimization problems $\min_{x\in\cC} f(x)$ over the feasible set $\cC$. Motivated by analogies to non-smooth mechanics, the algorithm possesses the following distinctive feature.
Unlike projection-based methods that project each iterate onto the feasible set, CGD projects the update direction (velocity) onto a local and sparse velocity polytope
\begin{equation*}
    V_\alpha(x) = \{v\in\bR^d \,|\, \alpha g_i(x) + \nabla g_i(x)^\top v \leq 0, \;\forall i\in I_x \},
\end{equation*}
where $I_x=\{i\in[m] \,|\, g_i(x) \geq 0\}$ is the set of active constraints and $\alpha>0$ controls the tradeoff between optimizing the objective and feasibility. At each iteration, CGD takes the update
\begin{align*}
    & v_t = \argmin{v\in V_\alpha(x_t)} \, \frac{1}{2}\norm{v + \nabla f(x_t)}^2, \\
    & x_{t+1} = x_t + \eta v_t.
\end{align*}
The development of its theoretical understanding centers around the continuous-time case for constrained gradient flow \citep{muehlebach2022constraints}, with extensions to accelerated methods \citep{muehlebach2023accelerated} and online, stochastic, and nonconvex minimization settings \citep{kolev2023online, schechtman2022first, schechtman2023Stochastic}. Recently, convergence guarantees of the discrete-time algorithm are established for smooth strongly-convex minimization problems in \citet{muehlebach2022constraints} and nonsmooth convex minimization problems in \citet{kolev2023online}.

\section{The Constrained Gradient Method (CGM)} \label{sec:cgm}

A standard algorithm in the literature on monotone variational inequality problems is the projected gradient method, which requires access to a projection oracle and performs the following updates at each iteration $t=0, 1,\ldots, T-1$:
\begin{equation*}
    x_{t+1} = \argmin{x\in\cC} \,\frac{1}{2}\norm{x - (x_t - \eta F(x_t))}^2, 
\end{equation*}
where $\eta>0$ is the stepsize. The projection step onto the feasible set involves solving a constrained optimization problem. When the feasible set $\cC$ does not have a simple structure, the above procedure is not always efficient and implementable. 

Instead, we propose the constrained gradient method (CGM; see Algorithm \ref{algo:cgd-vip}) that takes insights from \citet{muehlebach2022constraints} to alternatively project the velocity (search direction $-F(x_t)$) onto a local, sparse, and linear approximation of the feasible set.
\begin{algorithm}
    \caption{Constrained Gradient Method (CGM)}
    \begin{algorithmic}[1]
        \Require Initialization $x_0\in\cC$, stepsize $\{\eta_t\}_{t=0}^{T-1}>0$, parameter $\alpha>0$, precision $\epsilon>0$.
        \For{$t=0,1,\cdots,T-1$}
            \State Build the set of active constraints with the auxiliary constraint $g_{m+1}(x)=\norm{x}^2 - D^2$,
            \begin{equation*}
                I_{x_t} = \{i\in[m+1] \,|\, g_i(x_t) \geq 0\}.
            \end{equation*}
            
            \State Construct the velocity polytope
            \begin{equation*}
                V_\alpha(x_t) = \{v\in\bR^d \,|\, \alpha g_i(x_t) + \nabla g_i(x_t)^\top v \leq 0, \, \forall\, i\in I_{x_t}\}.
            \end{equation*}
            
            \State Solve the quadratic program
            \begin{equation*}
                v_t \approx \argmin{v\in V_\alpha(x_t)}\, \frac{1}{2}\norm{v+F(x_t)}^2,
            \end{equation*}
            such that $v_t\in V_\alpha(x_t)$ and
            \begin{equation*}
                (v_t + F(x_t))^\top (v_t  - v) \leq\frac{\epsilon}{2}, \quad \forall\, v\in V_\alpha(x_t). 
            \end{equation*}
            
            \State Update the parameter
            \begin{equation*}
                x_{t+1} = x_t + \eta_t v_t.
            \end{equation*}
        \EndFor
        \Ensure $\bar x_T:=(1/T)\sum_{t=0}^{T-1}x_t$ (monotone) or $\bar x_T:=(2/T(T-1))\sum_{t=0}^{T-1}tx_t$ (strongly-monotone).
    \end{algorithmic}
    \label{algo:cgd-vip}
\end{algorithm}
Algorithm \ref{algo:cgd-vip} requires solving a sequence of quadratic programs, which can be easily implemented and is more efficient compared to the evaluation of projection oracles with general (nonlinear) constraints. Here, the velocity polytope $V_\alpha(x_t)$ only involves active constraints, i.e., constraints that are not strictly satisfied, and thus the quadratic program is sparse and efficiently solvable even for large scale problems if the number of active constraints is small.
Since $V_\alpha(x_t)$ can be understood as an extension of the tangent cone to infeasible points, our algorithm shares similarities with the method of feasible directions \citep{nocedal1999numerical, bazaraa2013nonlinear} that linearize the constraint functions and use linear or quadratic programming to identify update directions that ensure feasibility of the iterates. In contrast, CGM only involves active constraints and allows iterates to be infeasible.

In Algorithm \ref{algo:cgd-vip}, we allow for an inexact solution of the quadratic program, where the stopping criteria can be satisfied using various methods. More details on how this quadratic program can be solved are discussed in Section \ref{sec:qp}.
The auxiliary constraint $g_{m+1}(x)=\norm{x}^2-D^2$ is used to ensure that both the iterate $x_t$ and the projected velocity $v_t$ are bounded. It is added specifically to simplify the theoretical analysis and will not be used in actual implementations.
The output of Algorithm \ref{algo:cgd-vip} is the average of iterates, which is a standard strategy in the literature \citep{nesterov2013introductory, bubeck2015convex}.

\section{Convergence Analysis} \label{sec:theory}

In this section, we provide a convergence analysis of Algorithm \ref{algo:cgd-vip} for two settings: $(i)$ when operator $F$ is monotone, and $(ii)$ when operator $F$ is strongly-monotone. We also discuss improved results for a special instance with functional constrained convex minimization problems. 

\subsection{The Monotone Setting} 

We first consider the constrained monotone variational inequality problems \eqref{eq:vip} and make the following two standard assumptions on the operator $F$ and the feasible set $\cC$.

\begin{assumption}
    The operator $F(x)$ is continuous and monotone on $\bR^d$. Moreover, its norm is upper bounded by $L_F$, i.e., $\norm{F(x)}\leq L_F$, $\forall x\in\bR^d$.
    \label{asp:vip}
\end{assumption}

\begin{assumption}
    Each constraint function $g_i(x)$ is convex, $L_g$-Lipschitz and $\ell_g$-smooth on $\bR^d$ for each $i\in[m]$. The feasible set $\cC$ is non-empty and contained in an Euclidean ball with radius $D$, i.e., $\norm{x}\leq D$, $\forall\, x\in\cC$.
    \label{asp:constraint}
\end{assumption}

Assumption \ref{asp:vip} is standard in the analysis of monotone variational inequality problems \citep{nemirovski2004prox, juditsky2011solving}. Assumption \ref{asp:constraint} ensures that the feasible set is convex and compact, which guarantees a nonempty and compact solution set \citep{facchinei2003finite}. The smoothness and Lipschitzness assumptions on the constraint functions are also frequently used in previous works on functional constrained problems \citep{lan2020algorithms, boob2023stochastic, boob2023first}. Note that both assumptions on $F(x)$ and $g_i(x)$ are made on the entire domain $\bR^d$, as opposed to the usual case for projection-based methods where the properties only hold on the feasible set $\cC$. This is due to the fact that the trajectory of our methods is not guaranteed to be always feasible. However, as we explain in Lemma \ref{lm:bounded}, the iterates of our algorithms remain in an Euclidean ball with a radius depending on $D$ and $L_F$. A proof is provided in Appendix \ref{app:convex}. As a result, all properties in Assumptions \ref{asp:vip} and \ref{asp:constraint} can be relaxed to only hold on this ball. We keep the current statement for simplicity of the results.

\begin{lemma}
    Let $\gamma>1$. Under Assumptions \ref{asp:vip} and \ref{asp:constraint}, Algorithm \ref{algo:cgd-vip} with $\epsilon\leq 4L_F^2$ and the choice of $\alpha>0$ such that $\alpha\eta_t\leq (\gamma - 1)/(\gamma + 1), \forall\, t=0,1,\cdots,T-1$ satisfies $\forall t=0,1,\cdots,T-1$,
    \begin{align*}
        & \norm{x_t}^2 \leq \gamma D^2 + \gamma(\gamma - 1)\left(D+\frac{4L_F}{\alpha}\right)^2, \\
        & \norm{v_t}^2 \leq (\gamma + 1)\alpha^2 D^2 + \gamma(\gamma + 1) \alpha^2 \left(D+\frac{4L_F}{\alpha}\right)^2.
    \end{align*}
    \label{lm:bounded}
\end{lemma}

Note that in Algorithm \ref{algo:cgd-vip}, we require the initial point to be feasible, namely, $x_0\in\cC$. However, according to Lemma \ref{lm:bounded}, this requirement can be relaxed to $\norm{x_0}^2 \leq \gamma D^2 + \gamma(\gamma - 1)(D+4L_F/\alpha)^2$, indicating that $x_0$ is allowed to be infeasible. We keep the current statement to simplify the presentation. With the fact that $\norm{v_t}$ is bounded, we can use standard techniques for the convergence analysis of CGM, as summarized in Theorem \ref{thm:convex}.

\begin{theorem}
    Under Assumptions \ref{asp:vip} and \ref{asp:constraint}, for $\epsilon\leq 4L_F^2$, Algorithm \ref{algo:cgd-vip} with stepsize $\eta_t\equiv\eta=D/(5L_F\sqrt{2T})$ and $\alpha=L_F/D$ satisfies
    \begin{equation*}
        F(x)^\top (\bar x_T - x) \leq \frac{10\sqrt{2}L_FD}{\sqrt{T}} + \frac{\epsilon}{2}, \quad\forall\, x\in\cC.
    \end{equation*}
    For the constraint violation, we have that for each $i\in[m+1]$ and $t=0,1,\cdots,T-1$,
    \begin{equation*}
        g_i(x_t) \leq \frac{\sqrt{2}D\max\{L_g, 5\ell_g D\}}{\sqrt{T}}.
    \end{equation*}
    \label{thm:convex}
\end{theorem}

\begin{proof}
    We first show the optimality guarantees. For each $i\in I_{x_t}$ such that $g_i(x_t)\geq 0$ is active, we conclude that $\forall x\in\cC$,
    \begin{align*}
        g_i(x_t) + \nabla g_i(x_t)^\top (x - x_t)
        & \leq
        g_i(x) \\
        & \leq
        0 \\
        & \leq
        g_i(x_t),
    \end{align*}
    where we have used the assumption that $g_i(x)$ is convex. This gives us $\nabla g_i(x_t)^\top(x - x_t) \leq 0, \forall x\in\cC$. By the definition of $V_\alpha(x_t)=\{v\in\bR^d \,|\, \alpha g_i(x_t) + \nabla g_i(x_t)^\top v\leq 0, \forall\,i\in I_{x_t}\}$, we obtain that
    \begin{equation*}
        (v + x - x_t) \in V_\alpha(x_t), \quad \forall\, x\in\cC, \; \forall\, v\in V_\alpha(x_t).
    \end{equation*}
    The quadratic programming solver guarantees that $v_t\in V_\alpha(x_t)$ and $(v_t + F(x_t))^\top(v_t - v) \leq \epsilon/2, \forall\, v\in V_\alpha(x_t)$. For any $x\in\cC$, we set $v = v_t + x - x_t \in V_\alpha(x_t)$ and get that
    \begin{equation}
        \begin{split}
            F(x_t)^\top(x_t - x)
            & \leq
            v_t^\top(x - x_t) + \frac{\epsilon}{2} \\
            & =
            \frac{1}{\eta_t}(x_{t+1} - x_t)^\top(x - x_t) + \frac{\epsilon}{2} \\
            & =
            \frac{1}{2\eta_t}\norm{x - x_t}^2 - \frac{1}{2\eta_t}\norm{x - x_{t+1}}^2 + \frac{1}{2\eta_t}\norm{x_{t+1} - x_t}^2 + \frac{\epsilon}{2} \\
            & =
            \frac{1}{2\eta_t}\norm{x - x_t}^2 - \frac{1}{2\eta_t}\norm{x - x_{t+1}}^2 + \frac{\eta_t}{2}\norm{v_t}^2 + \frac{\epsilon}{2}.
        \end{split}
        \label{eq:gap-onestep}
    \end{equation}
    We invoke Lemma \ref{lm:bounded} with $\gamma=3/2$ and $\epsilon\leq 4L_F^2$, and set $\eta_t=\eta=D/(5\sqrt{2} L_F\sqrt{T})$ and $\alpha=L_F/D$ such that $\alpha\eta_t=1/(5\sqrt{2T})<1/5$. This yields
    \begin{align*}
        \norm{v_t}^2
        & \leq
        \frac{5}{2}\alpha^2 D^2 + \frac{15}{4}\alpha^2\left(D + \frac{4L_F}{\alpha}\right)^2 \\
        & <
        100L_F^2.
    \end{align*}
    Summing up from $t=0$ to $T-1$ and dividing both sides by $T$, we obtain that $\forall\, x\in\cC$,
    \begin{align*}
        \frac{1}{T}\sum_{t=0}^{T-1} F(x_t)^\top(x_t - x)
        & \leq
        \frac{\norm{x-x_0}^2}{2\eta T} + 50 \eta L_F^2 + \frac{\epsilon}{2} \\
        & \leq
        \frac{10\sqrt{2} L_F D}{\sqrt{T}} + \frac{\epsilon}{2}.
    \end{align*}
    When $F(x)$ is monotone, it holds that $(F(x_t) - F(x))^\top(x_t - x) \geq 0$, and we have $\forall\,x\in\cC$,
    \begin{align*}
        F(x)^\top (\bar x_T - x)
        & =
        \frac{1}{T}\sum_{t=0}^{T-1} F(x)^\top (x_t - x) \\
        & \leq
        \frac{1}{T}\sum_{t=0}^{T-1} F(x_t)^\top(x_t - x) \\
        & \leq
        \frac{10 \sqrt{2} L_F D}{\sqrt{T}} + \frac{\epsilon}{2}.
    \end{align*}

    We now show the feasibility guarantees using induction. The base case is true since $x_0\in\cC$ and then $g_i(x_0)\leq 0$ for each $i\in[m+1]$. Assuming the claim holds for some $k\geq 0$, i.e., $g_i(x_k)\leq \sqrt{2} D\max\{L_g, 5\ell_g D\}/\sqrt{T}, \forall\,i\in[m+1]$, we now show the same is true for $k+1$. We consider the following two cases.

    $(a)$ For each $i\notin I_{x_k}$, we know $g_i(x_k)<0$. Applying Lemma \ref{lm:bounded}, by convexity and Lipschitz continuity of $g_i(x)$, we have
    \begin{equation}
        \begin{split}
            g_i(x_{k+1})
            & \leq
            g_i(x_k) + \nabla g_i(x_{k+1})^\top(x_{k+1} - x_k) \\
            & <
            \eta \nabla g_i(x_{k+1})^\top v_k \\
            & \leq
            10\,\eta L_g L_F \\
            & =
            \frac{\sqrt{2}L_g D}{\sqrt{T}}.
        \end{split}
        \label{eq:fea-case(i)}
    \end{equation}
    
    $(b)$ For all $i\in I_{x_k}$, we know $g_i(x_k)\geq0$ and $\eta\alpha g_i(x_k) + \nabla g_i(x_k)^\top (x_{k+1} - x_k)\leq 0$ from the construction of $V_\alpha(x_k)$. By $\ell_g$-smoothness of $g_i(x)$, we have that
    \begin{equation}
        \begin{split}
            g_i(x_{k+1})
            & \leq
            g_i(x_k) + \nabla g_i(x_k)^\top (x_{k+1} - x_k) + \frac{\ell_g}{2}\norm{x_{k+1} - x_k}^2 \\
            & \leq
            (1 - \alpha\eta) g_i(x_k) + 50\,\ell_g\eta^2L_F^2 \\
            & \leq
            \frac{\sqrt{2}D\max\{L_g, 5\ell_g D\}}{\sqrt{T}} + \frac{\ell_g D^2}{T} - \frac{D\max\{L_g, 5\ell_g D\}}{5T} \\
            & \leq
            \frac{\sqrt{2}D\max\{L_g, 5\ell_g D\}}{\sqrt{T}}.
        \end{split}
        \label{eq:fea-case(ii)}
    \end{equation}
    For both cases, we are able to show that the claim is true for $k+1$. As a result, the theorem holds true for every $t$, which completes the proof. 
\end{proof}

The above theorem implies that for CGM to achieve a weak $\epsilon$-solution, the number of queries to the operator $F(x)$ is at most $\cO(1/\epsilon^2)$,  and the number of calls to a quadratic programming solver is at most $\cO(1/\epsilon^2)$. Notably, the complexity of querying the operator $F$ matches that of projection-based methods \citep{nesterov2013introductory} or Frank-Wolfe type methods \citep{thekumparampil2020projection}, known to be optimal even for simple constraint sets according to the lower complexity bound (Theorem 3.2.1 in \citet{nesterov2013introductory}). 

In stark contrast, the convergence analysis of previous primal-dual methods relies on the information about the optimal Lagrange multipliers.
For instance, the ADMM-based interior point method \citep{yang2023solving, chavdarova2024a} achieves strong $\epsilon$-solutions with the same $\cO(1/\epsilon^2)$ queries to $F$ and requires $\cO(1/\epsilon^2)$ calls to a program that solves a sequence of strongly-convex minimization problems. However, their results require the strong assumption that either $F$ is strictly-monotone or one of the $g_i(x)$ is strictly-convex. The constraint extrapolation (ConEx) method \citep{boob2023first, boob2023stochastic} achieves a weak $\epsilon$-solution as in Definition \ref{def:opt} with $\cO(1/\epsilon^2)$ calls to the operator $F$. However, they also require the existence and boundedness of the optimal Lagrange multipliers.
To the best of our knowledge, CGM is the first primal method that achieves the optimal complexity on queries to $F$ for functional constrained monotone variational inequality problems without any information on the optimal Lagrange multipliers.

\subsection{The Strongly-Monotone Setting}

Next, we extend the convergence analysis of CGM (Algorithm \ref{algo:cgd-vip}) to the strongly-monotone setting, as stated in the theorem below.

\begin{theorem}
    Let $\gamma>1$ and let $F(x)$ be $\mu$--strongly-monotone on $\bR^d$. Under Assumptions \ref{asp:vip} and \ref{asp:constraint}, Algorithm \ref{algo:cgd-vip} with $T\geq2$, $\epsilon\leq 4L_F^2$, stepsize $\eta_t=1/\mu (t+1)$, and $\alpha=\mu(\gamma-1)/(\gamma+1)$ satisfies
    \begin{equation*}
        F(x)^\top (\bar x_T - x) \leq \frac{\mu M^2}{T-1} + \frac{\epsilon}{2}, \quad\forall x\in\cC,
    \end{equation*}
    where $M:=2(\gamma+1)(D+2L_F/\mu)$. For the constraint violation, we have $\forall i\in[m+1]$,
    \begin{equation*}
        g_i(\bar x_T) \leq \frac{\max\{12ML_g, 6\,\ell_g M^2\} + 6\,\ell_g M^2\,\zeta(1+2/(\gamma+1))}{(T+1)^{1-2/(\gamma+1)}},
    \end{equation*}
    where $\zeta(p):=\sum_{n=1}^\infty 1/n^p$ is the Riemann zeta function for $p>1$.
    \label{thm:strongly-convex}
\end{theorem}

\begin{proof}
    By applying the same reasoning as in \eqref{eq:gap-onestep}, see the proof of Theorem \ref{thm:convex}, we have that $\forall\,x\in\cC$,
    \begin{equation*}
        F(x_t)^\top(x_t - x) \leq \frac{1}{2\eta_t}\norm{x - x_t}^2 - \frac{1}{2\eta_t}\norm{x - x_{t+1}}^2 + \frac{\eta_t}{2}\norm{v_t}^2 + \frac{\epsilon}{2}.
    \end{equation*}
    When $F(x)$ is $\mu$--strongly-monotone, i.e., $(F(x_t) - F(x))^\top(x_t - x) \geq \mu\norm{x_t - x}^2$, we obtain that $\forall\,x\in\cC$,
    \begin{equation}
        \begin{split}
            F(x)^\top (x_t - x)
            & \leq
            F(x_t)^\top (x_t - x) - \mu\norm{x_t - x}^2 \\
            & \leq
            \left(\frac{1}{2\eta_t} - \mu\right)\norm{x - x_t}^2 - \frac{1}{2\eta_t}\norm{x - x_{t+1}}^2 + \frac{\eta_t}{2}\norm{v_t}^2 + \frac{\epsilon}{2}.
        \end{split}
        \label{eq:pf-sc-onestep}
    \end{equation}
    We set $\eta_t = 1/(\mu(t+1))$ and $\alpha=\mu(\gamma - 1)/(\gamma + 1)$ such that $\alpha\eta_t\leq (\gamma - 1)/(\gamma + 1)$ for every $t=0,1\cdots,T-1$ and constant $\gamma>1$. By Lemma \ref{lm:bounded}, we know that
    \begin{align*}
        \norm{v_t}^2
        & \leq
        \frac{1}{\gamma+1} \left[
            (\gamma-1)^2\mu^2D^2 + \gamma\Big((\gamma-1)\mu D + 4(\gamma+1)L_F\Big)^2
        \right] \\
        & <
        4(\gamma + 1)^2(\mu D + 2L_F)^2 \\
        & :=
        \mu^2 M^2,
    \end{align*}
    where we let $M=2(\gamma+1)(D + 2L_F/\mu)$ for simiplicity of the notation. Inserting the value of $\eta_t$ and $\norm{v_t}^2\leq\mu^2M^2$, \eqref{eq:pf-sc-onestep} becomes $\forall\,x\in\cC$,
    \begin{equation*}
        F(x)^\top (x_t - x) \leq
        \frac{\mu(t-1)}{2}\norm{x - x_t}^2 - \frac{\mu(t+1)}{2}\norm{x - x_{t+1}}^2 + \frac{\mu M^2}{2(t+1)} + \frac{\epsilon}{2}.
    \end{equation*}
    By multiplying both sides by $t\geq0$, we obtain
    \begin{equation*}
        t\,F(x)^\top (x_t - x) \leq
        \frac{\mu(t-1)t}{2}\norm{x - x_t}^2 - \frac{\mu\, t(t+1)}{2}\norm{x - x_{t+1}}^2 + \frac{\mu M^2}{2} + \frac{t\,\epsilon}{2}.
    \end{equation*}
    We recall that $\bar x_T=\sum_{t=0}^{T-1} tx_t / (\sum_{t=0}^{T-1} t)$. Summing up from $t=0$ to $T-1$ and dividing both sides by $\sum_{t=0}^{T-1} t$, we then have that $\forall\,x\in\cC$,
    \begin{align*}
        F(x)^\top (\bar x_T - x)
        & =
        \frac{2}{T(T-1)}\sum_{t=0}^{T-1} t\,F(x)^\top (x_t - x) \\
        & \leq
        \frac{2}{T(T-1)}\sum_{t=0}^{T-1} \frac{\mu M^2}{2} + \frac{2}{T(T-1)}\sum_{t=0}^{T-1} \frac{t\,\epsilon}{2} \\
        & \leq
        \frac{4(\gamma+1)^2\mu (D + 2L_F/\mu)^2}{T-1} + \frac{\epsilon}{2}.
    \end{align*}

    We proceed by showing the feasibility guarantees. For $t=0$, we have that $g_i(x_0)\leq 0$ for each $i\in[m+1]$ since $x_0\in\cC$. For iteration $t+1$ with $t\geq0$, we consider the following cases.

    $(a)$ When $i\notin I_{x_t}$, we know $g_i(x_t)<0$. Applying the same argument as \eqref{eq:fea-case(i)}, we have
    \begin{equation}
        \begin{split}
            g_i(x_{t+1})
            & <
            \eta_t \nabla g_i(x_{t+1})^\top v_t \\
            & \leq
            \frac{M L_g}{t+1}.
        \end{split}
        \label{eq:pf-fea-<0-onestep}
    \end{equation}
    
    $(b)$ When $i\in I_{x_t}$, we know $g_i(x_t)\geq0$ and $\eta_t\alpha g_i(x_t) + \nabla g_i(x_t)^\top (x_{t+1} - x_t)\leq 0$. By a similar argument to \eqref{eq:fea-case(ii)}, we have that
    \begin{equation}
        \begin{split}
            g_i(x_{t+1})
            & \leq
            (1 - \alpha\eta_t) g_i(x_t) + \frac{\ell_g}{2}\eta_t^2\,\mu^2M^2 \\
            & =
            \left(1 - \frac{c_\gamma}{t+1}\right) g_i(x_t) + \frac{c_M}{(t+1)^2},
        \end{split}
        \label{eq:pf-fea-onestep}
    \end{equation}
    where we let $c_\gamma=(\gamma - 1)/(\gamma + 1)$ and $c_M=\ell_gM^2/2$ for simplicity. Note that $0<c_\gamma<1$ when $\gamma>1$ and $c_\gamma$ approaches 1 in the limit $\gamma\to\infty$. By Lemma 4 in Chapter 2 of \citet{polyak1987introduction}, it is only possible to show that $g_i(x_t)=\cO(1/t^{c_\gamma})$, a rate slower than $\cO(1/t)$ unless $\gamma\to\infty$. We then formally obtain the guarantee on $g_i(x_{t+1})$. Multiplying $(t+2)^{c_\gamma}$ on both sides of \eqref{eq:pf-fea-onestep}, we have that
    \begin{equation}
        \begin{split}
            (t+2)^{c_\gamma} g_i(x_{t+1})
            & \leq
            \left(1 - \frac{c_\gamma}{t+1}\right) (t+2)^{c_\gamma} g_i(x_t) + \frac{c_M (t+2)^{c_\gamma}}{(t+1)^2} \\
            & =
            \left(1 - \frac{c_\gamma}{t+1}\right) \left(1 + \frac{1}{t+1}\right)^{c_\gamma} (t+1)^{c_\gamma} g_i(x_t) + \frac{c_M}{(t+1)^{2-c_\gamma}}\left(1 + \frac{1}{t+1}\right)^{c_\gamma} \\
            & \leq
            \left(1 - \frac{c_\gamma}{t+1}\right) \left(1 + \frac{c_\gamma}{t+1}\right) (t+1)^{c_\gamma} g_i(x_t) + \frac{c_M}{(t+1)^{2-c_\gamma}}\left(1 + \frac{c_\gamma}{t+1}\right) \\
            & \leq
            (t+1)^{c_\gamma} g_i(x_t) + \frac{2\,c_M}{(t+1)^{2-c_\gamma}},
        \end{split}
        \label{eq:pf-fea-onestep-trans}
    \end{equation}
    where we use $g_i(x_t)\geq0$ and the mean value theorem for the function $x^{c_\gamma}$ on the interval $(1, 1+1/(t+1))$ such that $(1+1/(t+1))^{c_\gamma}\leq 1+c_\gamma/(t+1)$ since $0<c_\gamma<1$. Let $\kappa(t)=\max\{0\leq s<t \,|\, g_i(x_s) \leq 0\}$, i.e., the last iterate that is feasible for constraint $i$. Note that $\kappa(t)\geq0$ must exist for every $t\geq1$ given that $g_i(x_0)\leq0$. Solving the recursion in \eqref{eq:pf-fea-onestep-trans}, we have that
    \begin{align*}
        (t+2)^{c_\gamma} g_i(x_{t+1})
        & \leq
        (\kappa(t+1)+2)^{c_\gamma} g_i(x_{\kappa(t+1)+1}) + \sum_{s=\kappa(t+1)+1}^t \frac{2\,c_M}{(s+1)^{2-c_\gamma}} \\
        & \leq
        (\kappa(t+1)+2)^{c_\gamma} g_i(x_{\kappa(t+1)+1}) + \sum_{n=1}^\infty \frac{2\,c_M}{n^{2-c_\gamma}} \\
        & =
        (\kappa(t+1)+2)^{c_\gamma} g_i(x_{\kappa(t+1)+1}) + 2\,c_M\,\zeta(2-c_\gamma),
    \end{align*}
    where we let $\sum_{\kappa(t+1)+1}^t=0$ if $\kappa(t+1)=t$. The series $\sum_{n=1}^\infty 1/n^{2-c_\gamma}$ converges and is finite since $2-c_\gamma>1$, and $\zeta(p)=\sum_{n=1}^\infty 1/n^p$ is the Riemann zeta function. Since $g_i(x_{\kappa(t+1)})\leq 0$, applying either \eqref{eq:pf-fea-<0-onestep} for the strictly feasible case or \eqref{eq:pf-fea-onestep} otherwise, we conclude
    \begin{align*}
        (\kappa(t+1)+2)^{c_\gamma} g_i(x_{\kappa(t+1)+1})
        & \leq
        (\kappa(t+1)+2)^{c_\gamma} \max\left\{\frac{M L_g}{\kappa(t+1)+1}, \frac{c_M}{(\kappa(t+1) + 1)^2}\right\} \\
        & =
        \left(1 + \frac{1}{\kappa(t+1) + 1}\right)^{c_\gamma} \max\left\{\frac{M L_g}{(\kappa(t+1)+1)^{1-c_\gamma}}, \frac{c_M}{(\kappa(t+1) + 1)^{2-c_\gamma}}\right\} \\
        & \leq
        2 \max\left\{\frac{M L_g}{(\kappa(t+1)+1)^{1-c_\gamma}}, \frac{c_M}{(\kappa(t+1) + 1)^{2-c_\gamma}}\right\} \\
        & \leq
        2\max\{M L_g, c_M\}.
    \end{align*}
    This implies that
    \begin{equation*}
        g_i(x_{t+1}) \leq \frac{2\max\{M L_g, c_M\} + 2\,c_M\,\zeta(2-c_\gamma)}{(t+2)^{c_\gamma}},
    \end{equation*}
    which is also a valid upper bound of \eqref{eq:pf-fea-<0-onestep} since $c_\gamma<1$.
    
    Given the upper bound on each $g_i(x_t)$, by convexity of $g_i(x)$, we further have that
    \begin{align*}
        g_i(\bar x_T)
        & \leq
        \frac{2}{T(T-1)} \sum_{t=1}^{T-1} t\, g_i(x_t) \\
        & \leq
        \frac{4\max\{M L_g, c_M\} + 4\,c_M\,\zeta(2-c_\gamma)}{T(T-1)} \sum_{t=1}^{T-1} \frac{t}{(t+1)^{c_\gamma}} \\
        & \leq
        \frac{4\max\{M L_g, c_M\} + 4\,c_M\,\zeta(2-c_\gamma)}{T(T-1)} \sum_{t=1}^{T} t^{1-c_\gamma} \\
        & \leq
        \frac{4\max\{M L_g, c_M\} + 4\,c_M\,\zeta(2-c_\gamma)}{(T-1)} \left(\frac{T+1}{2}\right)^{1-c_\gamma} \\
        & \leq
        \frac{12\max\{M L_g, c_M\} + 12\,c_M\,\zeta(2-c_\gamma)}{(T+1)^{c_\gamma}},
    \end{align*}
    where we use $T\geq2$ and Jensen's inequality on the concave function $x^{1-c_\gamma}$.
\end{proof}

In the strongly-monotone setting, the convergence rate of CGM in terms of the optimality gap is $\cO(1/T)$, which matches the rate of projection-based methods \citep{nesterov2013introductory}. However, it is worth noting that the rate in terms of the constraint violation approaches $\cO(1/T)$ only  as $\gamma$ grows large enough. We leave the task of closing this gap to future work. As a comparison, the primal-dual method in \citet{yang2023solving} achieves $\cO(1/\sqrt{T})$ rate measured by the distance $\norm{\bar x_T - x^*}$.
If the operator is additionally Lipschitz, the same $\cO(1/\sqrt{T})$ rate is attained for a strong approximate solution.

\subsection{The Special Case with Convex Minimization}

In the following, we provide a special instance of the strongly-monotone variational inequality problems such that an $\cO(1/T)$ convergence rate on both the optimality gap and constraint violations can be attained through a different strategy of setting the stepsizes. To be specific, we consider the strongly-convex and smooth minimization problem with functional constraints, that is,
\begin{align*}
    & \min_{x\in\bR^d} \, f(x), \\
    & \text{s.t. } \quad g_i(x) \leq 0, \, \forall i\in[m].
\end{align*}
This is equivalent to the variational inequality problem associated with the operator $F(x):=\nabla f(x)$ and the feasible set $\cC=\{x\in\bR^d | g_i(x)\leq 0, \forall i\in[m]\}$. The guarantees of CGM in this setting are as follows.
Note that the guarantees in this case apply to the last iterate $x_T$.

\begin{theorem}
    Suppose $f(x)$ is $\mu$--strongly-convex, $\ell_f$-smooth, and $L_f$-Lipschitz, and  $g_i(x)$ is convex, $\ell_g$-smooth, and $L_g$-Lipschitz for each $i\in[m]$ such that the feasible set $\cC$ is contained in an Euclidean ball of radius $D$. Let $T$ be large enough such that $T\geq\max\{3,\kappa_f\}(\log T)$, where $\kappa_f:=\ell_f/\mu$ denotes the condition number.  Then CGM with $\epsilon=0$ for simplicity, stepsize $\eta=(\log T)/(\mu T)$, and $\alpha=\mu$ satisfies that
    \begin{equation*}
        f(x_T) - f(x^*) \leq \frac{f(x_0) - f(x^*)}{T},
    \end{equation*}
    where $x^*=\arg\min_{x\in\cC} f(x)$.
    For the constraint violation, we have that for every iteration $t$ and for each $i\in[m+1]$,
    \begin{equation*}
        g_i(x_t) \leq\frac{M\max\{2L_g, \ell_gM\}\log T}{2T},
    \end{equation*}
    where $M:=3(D+4L_f/\mu)$. 
    \label{thm:primal}
\end{theorem}

\begin{proof}
    We first show the guarantees on the objective function. For the velocity polytope $V_\alpha(x_t)$, we have that $\forall x\in\cC$ such that $g_i(x) \leq 0$,
    \begin{align*}
        \alpha g_i(x_t) + \alpha\nabla g_i(x_t)^\top (x - x_t)
        & \leq
        \alpha g_i(x) 
         \leq
        0,
    \end{align*}
    since each $g_i(x)$ is convex and $\alpha>0$. This implies that $\alpha(x - x_t)\in V_\alpha(x_t)$ for any $x\in\cC$. By the computation of $v_t$, we then have that $\norm{v_t + \nabla f(x_t)}^2\leq\norm{\alpha(x^*-x_t) + \nabla f(x_t)}^2$ for $x^*=\arg\min_{x\in\cC} f(x)$. Using smoothness of $f(x)$, we conclude
    \begin{align*}
        f(x_{t+1})
        & \leq
        f(x_t) + \nabla f(x_t)^\top(x_{t+1} - x_t) + \frac{\ell_f}{2}\norm{x_{t+1} - x_t}^2 \\
        & =
        f(x_t) + \eta\nabla f(x_t)^\top v_t + \frac{\ell_f}{2}\eta^2\norm{v_t}^2 \\
        & =
        f(x_t) + \frac{\eta}{2}\norm{v_t + \nabla f(x_t)}^2 - \frac{\eta}{2}\norm{\nabla f(x_t)}^2 - \frac{\eta}{2}(1-\eta\ell_f)\norm{v_t}^2 \\
        & \leq
        f(x_t) + \frac{\eta}{2}\norm{\alpha(x^* - x_t) + \nabla f(x_t)}^2 - \frac{\eta}{2}\norm{\nabla f(x_t)}^2 - \frac{\eta}{2}(1-\eta\ell_f)\norm{v_t}^2 \\
        & =
        f(x_t) + \frac{\eta}{2}\alpha^2\norm{x^* - x_t}^2 + \alpha\eta\nabla f(x_t)^\top(x^* - x_t) - \frac{\eta}{2}(1-\eta\ell_f)\norm{v_t}^2 \\
        & \leq
        f(x_t) - \alpha\eta (f(x_t) - f(x^*)) + \frac{\eta\alpha}{2}(\alpha - \mu)\norm{x^* - x_t}^2 - \frac{\eta}{2}(1-\eta\ell_f)\norm{v_t}^2,
    \end{align*}
    where we use the assumption that $f(x)$ is $\mu$--strongly-convex. When setting $\alpha \leq \mu$, we obtain
    \begin{equation*}
        f(x_{t+1}) - f(x^*) \leq (1 - \alpha\eta) (f(x_t) - f(x^*)) - \frac{\eta}{2}(1-\eta\ell_f)\norm{v_t}^2.
    \end{equation*}
    When $\eta\leq1/\ell_f$, we have that $1-\eta\ell_f\geq 0$ and $0<\alpha\eta\leq\mu/\ell_f\leq 1$, and thus
    \begin{align*}
        f(x_T) - f(x^*)
        & \leq
        (1 - \alpha\eta)^T (f(x_0) - f(x^*)) \\
        & \leq
        \exp(-\alpha\eta T) (f(x_0) - f(x^*)) \\
        & =
        \frac{f(x_0) - f(x^*)}{T},
    \end{align*}
    if $\eta=(\log T)/(\mu T)$ and $\alpha=\mu$.

    We then show the guarantees on the constraint violations by induction. Applying Lemma \ref{lm:bounded} with $\gamma=2$, $\eta=(\log T)/(\mu T)$ and $\alpha=\mu$, this gives that for every $t$, $\norm{v_t}^2\leq3\mu^2D^2 + 6\mu^2(D+4L_f/\mu)^2\leq\mu^2 M^2$. The base case is true since $x_0\in\cC$. Assuming the claim holds for some $k\geq0$, we prove that the same is true for $k+1$. For each $i\notin I_{x_k}$, applying the same analysis as \eqref{eq:fea-case(i)}, we have that
    \begin{align*}
        g_i(x_{k+1})
        & \leq
        \eta \nabla g_i(x_{k+1})^\top v_k \\
        & \leq
        \frac{L_g M\log T}{T}.
    \end{align*}
    For each $i\in I_{x_t}$, we apply the same argument as in \eqref{eq:fea-case(ii)}, which yields
    \begin{align*}
        g_i(x_{k+1})
        & \leq
        (1 - \alpha\eta) g_i(x_k) + \frac{\ell_g}{2}\eta^2\norm{v_k}^2 \\
        & \leq
        \frac{M\max\{2L_g, \ell_gM\}\log T}{2T} + \frac{\ell_g M^2 (\log T)^2}{2T^2} - \frac{M\max\{2L_g, \ell_gM\}(\log T)^2}{2T^2} \\
        & \leq
        \frac{M\max\{2L_g, \ell_gM\}\log T}{2T}.
    \end{align*}
    For both cases, we show that the claim holds for $k+1$, which completes the proof. Note that we also assume that $T$ is large enough such that $T\geq \max\{\kappa_f, 3\}(\log T)$ to ensure that $\eta\leq1/\ell_f$ and $\alpha\eta\leq1/3$ as required in the convergence analysis.
\end{proof}

We point out that the original analysis of CGM for strongly-convex and smooth minimization problems in \citet{muehlebach2022constraints} achieves $\cO(1/T)$ rate on the distance $\min_{0\leq t\leq T} \norm{x_t - x^*}^2$. However, they require the stepsize $\eta\leq2/(\ell^* + \mu)$, where $\ell^*$ is the smoothness parameter of the Lagrangian function $f(x)+\sum_{i\in[m]}\lambda^*_i g_i(x)$ and has explicit dependence on the optimal multipliers $\lambda^*_i$. Instead, we provide an improved analysis in Theorem \ref{thm:primal} where no prior knowledge on the optimal Lagrange multipliers is required.

\section{Solving the Quadratic Program in CGM} \label{sec:qp}

This section summarizes different approaches to solving the quadratic program in CGM (Algorithm \ref{algo:cgd-vip}), which determines the update direction. Notably, there are important examples where the solutions to the quadratic program can be derived in closed form, including the case involving only one constraint function and simplex constraints. Additional examples, including nonconvex constraints and constraints arising in optimal transport, can be found in \citet{ibrahim2023} and \citet{schechtman2023Stochastic}. The availability of closed-form solutions to the quadratic programs in CGM enables direct implementations that are as computationally efficient as unconstrained gradient methods.

\subsection{Single Constraint Function}

Let us consider the variational inequality problems \eqref{eq:vip} where the feasible set is described by a single constraint function $g(x)$:
\begin{equation*}
    \cC=\{x\in\bR^d \,|\, g(x)\leq 0\}.
\end{equation*}
The quadratic program in CGM reduces to a single linear constraint when $g(x)$ is active, and closed-form solutions can be achieved by solving the KKT system (details omitted). The resulting algorithm is summarized in Algorithm \ref{algo:cgm-1}.

\begin{algorithm}
    \caption{Constrained Gradient Method with One Constraint}
    \begin{algorithmic}[1]
        \Require Initialization $x_0$, stepsize $\eta>0$, parameter $\alpha>0$.
        \For{$t=0,1,\cdots,T-1$}
            \If{$g(x_t)\geq0$ and $\alpha g(x_t) - \nabla g(x_t)^\top F(x_t)\geq0$}
                \State $\lambda = \norm{\nabla g(x_t)}^{-2}(\alpha g(x_t) - \nabla g(x_t)^\top F(x_t))$.
            \Else
                \State $\lambda=0$.
            \EndIf
            \State $x_{t+1} = x_t - \eta F(x_t) - \eta\lambda\, \nabla g(x_t)$.
        \EndFor
        \Ensure $\bar x_T = (1/T)\sum_{t=0}^{T-1} x_t$.
    \end{algorithmic}
    \label{algo:cgm-1}
\end{algorithm}

Algorithm \ref{algo:cgm-1} has direct update steps that are easy to implement.  Interestingly, it coincides with the dynamical barrier approach proposed separately by \citet{gong2021automatic} for constrained minimization problems.

\subsection{The Simplex Constraints}

Although the simplex involves multiple constraints, the quadratic program in CGM has closed-form solutions. This allows CGM to perform direct update steps that can be computed in at most $\cO(d\log d)$ steps. As detailed in Appendix \ref{app:sp}, given an index set $N\subseteq[d]$ and a vector $q\in\bR^{d}$, the problem reduces to solving the quadratic program
\begin{align*}
    & \min_{p\in \bR^{d}} \, \frac{1}{2}\norm{p-q}^2, \quad \text{s.t. } \quad \sum_{i=1}^{d} p_i=1,  \quad p_i\geq 0, \forall i\in N,
\end{align*}
where $p_i$ denotes the $i$-th coordinate of $p$. This is similar to the projection problem onto the simplex, but we only restrict the coordinates in the set $N$ to be non-negative. In Appendix \ref{app:sp}, we develop a method proj$_v(q, N)$, summarized in Algorithm \ref{algo:proj-v}, analogous to the projection oracle onto the simplex \citep{duchi2008efficient} for solving the above problem.

\begin{algorithm}
    \caption{Velocity Projection Oracle (proj$_v$)}
    \begin{algorithmic}[1]
        \Require Vector $q\in\bR^{d}$, index set $N\subseteq[d]$ with size $0\leq n\leq d$.
        \State Compute $s_{\bar N} = \sum_{i=1, i\notin N}^{d} q_i$.
        \State Sort $\{q_i \,|\, i\in N\}$ into $r_1\geq r_2\geq \cdots \geq r_n$.
        \State Construct the set $J = \big\{1\leq j \leq n \,\big|\, r_j + \frac{1}{d-n+j}\big(1 - s_{\bar N} - \sum_{i=1}^j r_i\big) > 0\big\}$.
        \If{$J=\emptyset$}
            \State $\lambda=\frac{1}{d-n}(1-s_{\bar N})$.
        \Else
            \State Let $\rho=\max J$.
            \State $\lambda = \frac{1}{d-n+\rho} \left(1 - s_{\bar N} - \sum_{i=1}^\rho r_i\right)$.
        \EndIf
        \If{$i\in N$}
            \State $p_i = \max\{0, q_i + \lambda\}$.
        \Else
            \State $p_i = q_i + \lambda$.
        \EndIf
        \Ensure $p$.
    \end{algorithmic}
    \label{algo:proj-v}
\end{algorithm}

\begin{algorithm}
    \caption{Constrained Gradient Method with Simplex Constraint}
    \begin{algorithmic}[1]
        \Require Initialization $x_0$, stepsize $\eta>0$, parameter $\alpha>0$.
        \For{$t=0,1,\cdots,T-1$}
            \State $N_t = \{i\in[d] \,|\, x_{t,i}\leq 0\}$.
            \State $x_{t+1} = (1 - \alpha\eta) x_t + \alpha\eta\; \text{proj}_v\left(x_t - \frac{1}{\alpha} F(x_t), N_t\right)$.
        \EndFor
        \Ensure $\bar x_T = (1/T)\sum_{t=0}^{T-1} x_t$.
    \end{algorithmic}
    \label{algo:cgm-sp}
\end{algorithm}

The resulting algorithm implementing CGM is provided in Algorithm \ref{algo:cgm-sp}, where $x_{t,i}$ denotes the $i$-th coordinate of the iterate $x_t$. Compared to the projection-based methods that require sorting a vector of dimension $d$ \citep{duchi2008efficient}, Algorithm \ref{algo:cgm-sp} is provably more efficient when the velocity polytope is sparse, as the procedure in Algorithm \ref{algo:proj-v} only sorts a vector of dimension $n\leq d$. This manifests potential benefit of CGM even when the projection oracle is computable.

\subsection{The General Case}

For general problems with multiple constraints $\cC=\{x\in\bR^d \,|\, g_i(x)\leq 0, \forall i\in[m]\}$, closed-form updates of CGM may not always exist.
A simple solution, akin to how Polyak's switching gradient method \citep{polyak1967general, huang2023oracle} handles multiple constraints, is to replace $\cC$ with $\{x\in\bR^d \,|\, g(x)\leq 0\}$ for $g(x)=\max_{i\in[m]} g_i(x)$.
As our algorithm requires the constraint function to be smooth, a better choice is to use the log-sum-exp function $g(x)=\log(\sum_{i\in[m]} \exp(g_i(x)))$ as a smooth approximation to the maximum.
Algorithm \ref{algo:cgm-1} can be applied since now there is only one constraint function involved.
This strategy is not commonly preferred, as it overlooks the structure of the feasible set and the velocity polytope. Various methods \citep{muehlebach2022constraints, nishihara2015general} can then be applied to actually solve the quadratic program; see Section 7 in \citet{muehlebach2022constraints} for discussions. For example, the Frank-Wolfe method achieves the required guarantees in Algorithm \ref{algo:cgd-vip} with $\cO(1/\epsilon)$ calls to a linear programming solver \citep{jaggi2013revisiting}.

\section{Numerical Experiments} \label{sec:exp}

Numerical experiments on minimax problems with quadratic or simplex constraints are provided to evaluate the effectiveness of CGM. For all settings, we implement CGM with direct updates as discussed in Section \ref{sec:qp}.

\subsection{2D Examples with Ellipse Constraints}

We first provide the following two 2D examples with minimax objectives and ellipse constraints to illustrate the trajectory of our algorithm: the Forsaken game \citep{hsieh2021limits} and a toy GAN \citep{daskalakis2018training}. Both problems do not have constraints in their original formulation, and we add an ellipse constraint to test our algorithm. To be specific, the forsaken game has the objective
\begin{align*}
    & \min_{x\in\bR} \max_{y\in\bR} \; x(y - 0.45) + h(x) - h(y), \\
    & \text{s.t.} \quad x^2 + 4y^2 \leq 1,
\end{align*}
where $h(x)=x^2/4-x^4/2+x^6/6$. The toy GAN is constructed such that the generator tries to learn the unknown variance of data sampled from a Gaussian distribution, that is,
\begin{align*}
    & \min_{x\in\bR} \max_{y\in\bR} \; \bE_{u_1\sim \cN(0,1)} \left[y u_1^2\right] - \bE_{u_2\sim \cN(0,1)}\left[y x^2 u_2^2\right], \\
    & \text{s.t.} \quad x^2 + 4y^2 \leq 1.
\end{align*}
Both examples are known to exhibit limit cycles, posing challenges in the computation of equilibria. The projection operation onto general ellipsoid constraints does not have simple solutions, and previous methods rely on an iterative procedure for an approximation \citep{dai2006fast}. However, since the ellipse constraint only involves one constraint function, our algorithm has a simple update rule as summarized in Algorithm \ref{algo:cgm-1}.

In Figure \ref{fig:2d}, we plot the trajectory of CGM (Algorithm \ref{algo:cgm-1}) on the 2D examples. For both examples, we fix the stepsize $\eta$ to be 0.1, the number of iterations to be 64, and vary the parameter $\alpha$ as presented in the figure. The optimal solution is inside the feasible set for the Forsaken game, and exactly on the boundary for the toy GAN. For the toy GAN, we use 1000 independent samples from a zero-mean Gaussian random variable with unit variance for both $u_1$ and $u_2$ to estimate the gradients at each iteration. The figure illustrates that CGM exhibits different behaviors as $\alpha$ changes: a larger $\alpha$ results in more aggressive updates towards minimizing constraint violations, whereas a smaller $\alpha$ leads to trajectories that are predominantely guided by the objective function. CGM demonstrates the ability to approach the feasible set and attain convergence to the optimal solution, even when initiated from infeasible points.

\begin{figure}
    \centering
    \includegraphics[width=\textwidth]{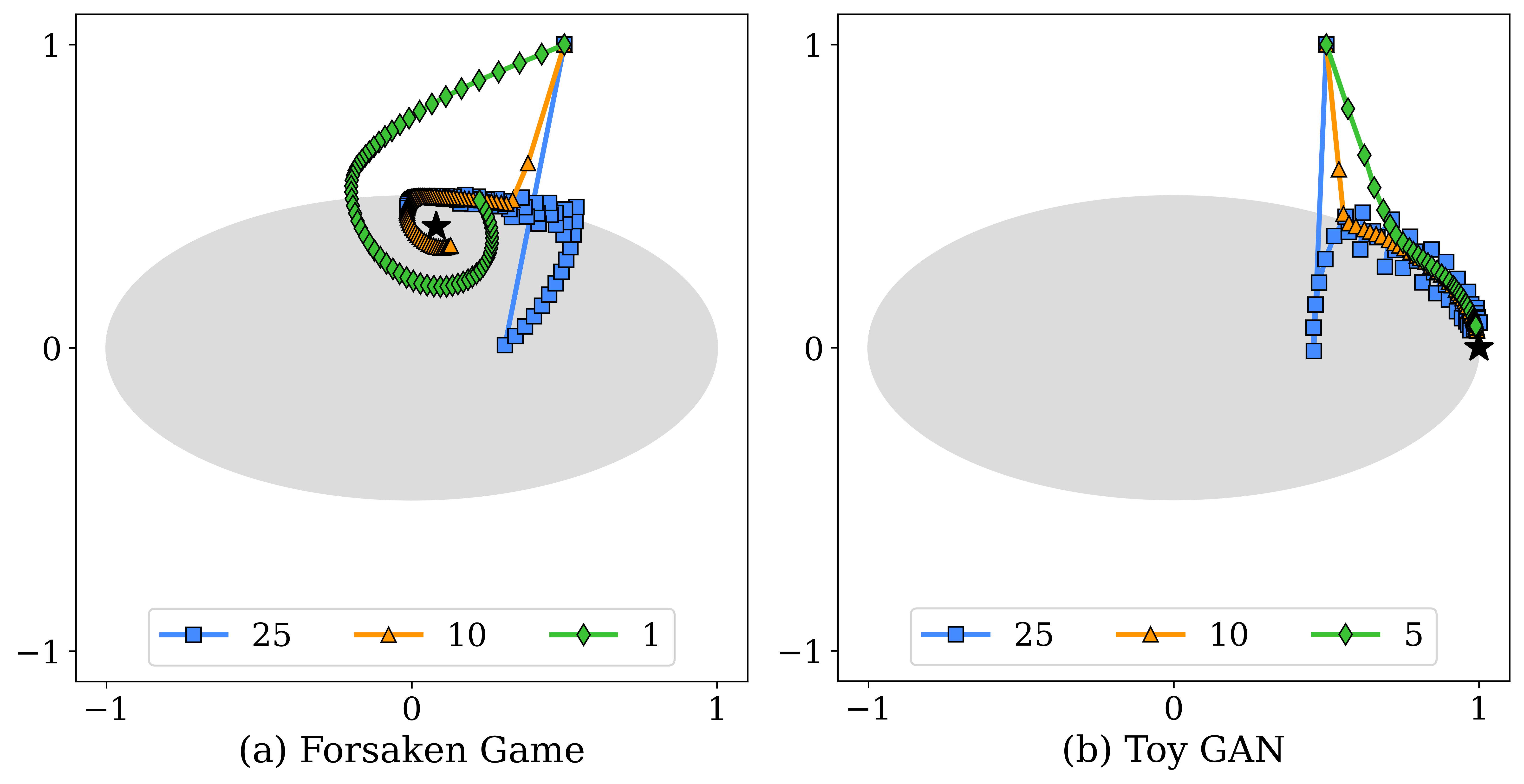}
    \caption{The trajectory of CGM on two 2D examples with varying $\alpha$. Each marker denotes the iterate at each iteration. The shaded area represents the feasible set and the black star denotes the optimal solution. The initialization point is (0.5,1) for both examples.}
    \label{fig:2d}
\end{figure}

\subsection{Matrix Game with Quadratic Constraints}

We then turn to problems with higher dimensions. The first example is the matrix game with a quadratic constraint, i.e.,
\begin{align*}
    & \min_{x\in\bR^d} \max_{y\in\bR^d} \; (x-a)^\top A y, \\
    & \text{s.t.} \quad \frac{1}{2} z^\top B z \leq c,
\end{align*}
where $z=(x,y)\in\bR^{2d}$ is the concatenation of $x$ and $y$. The problem is equivalent to the constrained variational inequality problem associated with $F(z)=(Ay, -A^\top (x-a))$ and $\cC=\{z\in\bR^{2d} \,|\, (1/2)z^\top B z \leq c\}$. Here, $A\in\bR^{d\times d}$ is a random matrix with each entry sampled from $\cN(0,1)$, $a\in\bR^d$ is a vector with each entry sampled from $\cN(0,0.1)$, $B\in\bR^{2d\times 2d}$ is a random positive definite matrix with each eigenvalue uniformly sampled from $[0.1,10]$, and $c>0$ is sampled uniformly from $[0.1,10]$. For any point $\hat z=(\hat x,\hat y)$, its optimality is measured by the (strong) gap $\max_{z\in\cC} F(\hat z)^\top (\hat z - z)$, effectively the same as $-\min_{z\in\cC} z^\top F(\hat z) = (2c\, F(\hat z)^\top B^{-1} F(\hat z))^{1/2}$ through solving the corresponding KKT system, and its feasibility is measured by the constraint violation $\max\{0, (1/2)\hat z^\top B \hat z - c\}$.

Although the constraint function has a simple quadratic form, it is challenging to derive the projection oracle. Its specific form depends on the spectrum of $B$ and often relies on matrix inversion and decomposition operations that are computationally expensive, except if $B$ is diagonal. Existing methods therefore rely on an iterative procedure to derive the projection oracle, see e.g., \citet{dai2006fast}. However, as there is only one constraint function, an efficient implementation of CGM exists; see Algorithm \ref{algo:cgm-1}.
Figure \ref{fig:quad} shows the performance of CGM (Algorithm \ref{algo:cgm-1}). The experiments are conducted with a dimension of $d=1000$, an iteration number of $T=1000$, stepsize $\eta=0.01$, $\alpha=50$, and a random initialization point with each coordinate sampled from the Gaussian $\cN(0,1)$.
CGM is able to reduce both the optimality gap and feasibility without knowing the projection oracle, as indicated by our theoretical analysis.

\begin{figure}
    \centering
    \includegraphics[width=\textwidth]{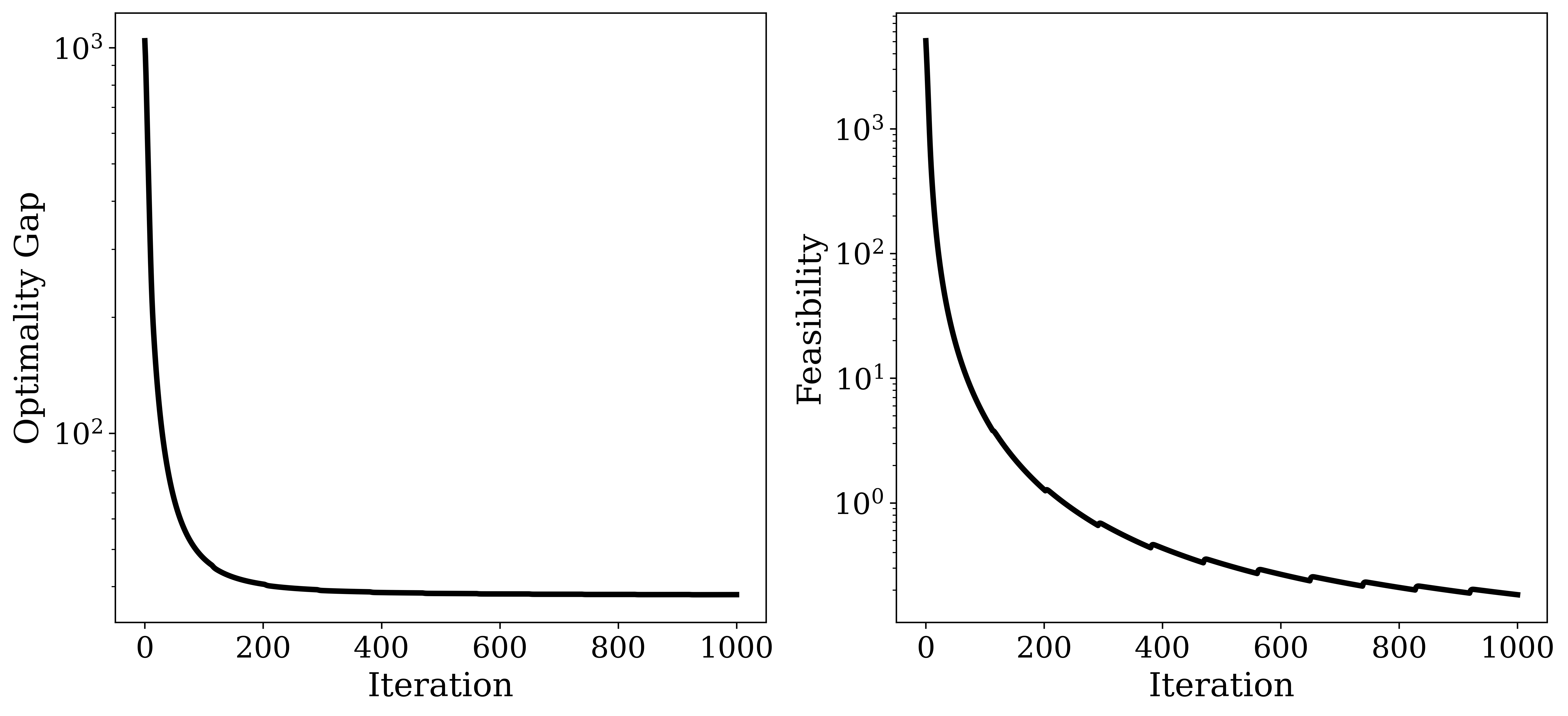}
    \caption{The optimality gap and feasibility of CGM on the matrix game with the quadratic constraint. Both are measured on the average iterates of the algorithm.}
    \label{fig:quad}
\end{figure}

\subsection{Matrix Game with Simplex Constraints}

\begin{figure}
    \centering
    \includegraphics[width=\textwidth]{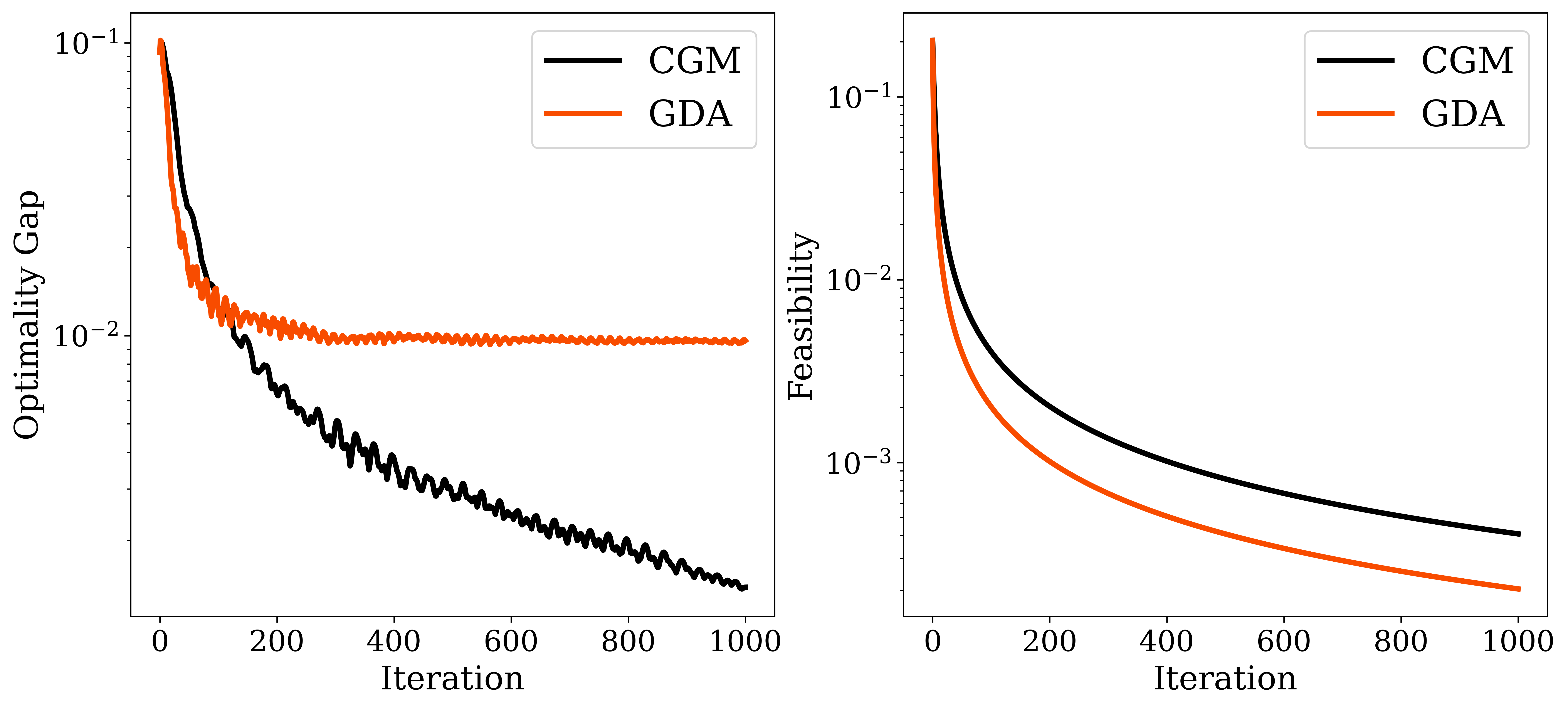}
    \caption{The optimality gap and feasibility of CGM (Algorithm \ref{algo:cgm-sp}) and projected gradient descent ascent (GDA) on the matrix game with the simplex constraint. Both are measured on the average iterates of the algorithm.}
    \label{fig:simplex}
\end{figure}

The last example we explore is a matrix game with simplex constraints. This gives rise to
\begin{align*}
    & \min_{x\in\bR^d} \max_{y\in\bR^d} \; x^\top A y, \\
    & \text{s.t.} \quad \sum_{i=1}^d x_i + \sum_{i=1}^d y_i = 1, \\
    & \phantom{\text{s.t.} \quad} x_i \geq 0, \, y_i \geq 0, \; \forall i\in[d],
\end{align*}
where $A\in\bR^{d\times d}$ is a random matrix with each entry sampled from $\cN(0,1)$. Let $z=(x,y)\in\bR^{2d}$. The problem is equivalent to the constrained variational inequality problem associated with the operator $F(z)=(Ay, -A^\top x)$ and the feasible set $\cC=\{z\in\bR^{2d} \,|\, \sum_{i=1}^{2d} z_i=1, z_i\geq 0, \forall i\in[2d]\}$.
For any point $\hat z=(\hat x,\hat y)$, its optimality is measured by the (strong) gap $\max_{z\in\cC} F(\hat z)^\top (\hat z - z) \leq \abs{\max (A^\top \hat x)_i} + \abs{\min (A \hat y)_i}$, and its feasibility is measured by the constraint violation $\max\{0, -\hat z_i, \abs{\sum_{i=1}^{2d} z_i-1}\}$.

Figure \ref{fig:simplex} shows the performance of CGM (Algorithm \ref{algo:cgm-sp}) and compares with projection-based method gradient descent ascent (GDA), which suggests that both the optimality gap and feasibility can be effectively reduced by CGM.
In the experiments for both CGM and GDA, we set the dimension $d=1000$, the number of iterations $T=1000$, stepsize $\eta=0.005$, and choose a random initialization point with each coordinate sampled from the Gaussian $\cN(0,1)$. The parameter $\alpha$ is set to be 100 for CGM. Since we start with an initial point that may be infeasible, the average iterates of GDA may not be feasible as well, even with projections applied at each iteration.

\section{Conclusion} \label{sec:conclude}

We present a primal method for solving monotone variational inequality problems with general functional constraints. Our algorithm achieves the same complexity on querying $F$ as projection-based methods and does not rely on any information about the optimal Lagrange multipliers unlike existing primal-dual methods.
Several interesting questions have yet to be explored. The current analysis of CGM (Algorithm \ref{algo:cgd-vip}) for the strongly-monotone setting achieves a rate slightly worse than $\cO(1/T)$ on the constraint violation, where $T$ is the number of iterations. It remains unknown whether such guarantees can be improved.
Recently, \citet{zamani2023exact} presented a refined analysis of the projected subgradient method, demonstrating last iterate convergence for nonsmooth convex optimization. Extending our analysis in this direction could be interesting.
An exploration into the monotone and Lipschitz case is also valuable, particularly to determine if the improved complexity $\cO(1/\epsilon)$ can be achieved, as observed in the context of projection-based methods, such as extragradient and optimistic gradient \citep{mokhtari2020convergence, mokhtari2020unified}.
Furthermore, extensions to stochastic settings \citep{juditsky2011solving, alacaoglu2022stochastic, lan2020algorithms} and non-monotone settings \citep{yang2020global, diakonikolas2021efficient} are left for future investigations.

\section*{Acknowledgements}

We thank the Max Planck ETH Center for Learning Systems and the German Research Foundation for the support.

\bibliography{ref}
\bibliographystyle{plainnat}

\newpage
\appendix

\section{Proof of Lemma \ref{lm:bounded}}
\label{app:convex}

\begin{proof}
    Since $\alpha>0$ and $g_i(x)$ is convex for any $i\in[m+1]$, we have that $\forall x\in\cC$,
    \begin{align*}
        \alpha g_i(x_t) + \alpha\nabla g_i(x_t)^\top (x - x_t)
        & \leq
        \alpha g_i(x) \\
        & \leq
        0.
    \end{align*}
    We recall that $V_\alpha(x_t)=\{v\in\bR^d \,|\, \alpha g_i(x_t) + \nabla g_i(x_t)^\top v \leq 0, \forall i\in I_{x_t}\}$. The equation above implies that $\alpha(x-x_t)\in V_\alpha(x_t)$ for any $x\in\cC$. By the guarantee that $(v_t+F(x_t))^\top (v_t - v)\leq\epsilon/2$, $\forall v\in V_\alpha(x_t)$, we know that for any $v\in V_\alpha(x_t)$,
    \begin{align*}
        \frac{1}{2}\norm{v_t + F(x_t)}^2  - \frac{1}{2}\norm{v + F(x_t)}^2
        & =
        (v_t+F(x_t))^\top (v_t - v) - \frac{1}{2}\norm{v_t - v}^2 \\
        & \leq
        \frac{\epsilon}{2}.
    \end{align*}
    Setting $v=\alpha(x-x_t)$ for any $x\in\cC$ and $\epsilon\leq 4L_F^2$, we have that $\forall x\in\cC$,
    \begin{align*}
        \norm{v_t + F(x_t)}
        & \leq
        \norm{\alpha(x-x_t) + F(x_t)} + \sqrt{\epsilon} \\
        & \leq
        \alpha\norm{x_t} + \alpha\norm{x} + \norm{F(x_t)} + 2L_F.
    \end{align*}
    By the assumptions that $\norm{F(x)}\leq L_F$ and $\norm{x}\leq D, \forall x\in\cC$, we have that   
    \begin{equation}
        \begin{split}
            \norm{v_t}
            & \leq
            \norm{v_t + F(x_t)} + \norm{F(x_t)} \\
            & \leq
            \alpha\norm{x_t} + \alpha\norm{x} + 2\norm{F(x_t)} + 2L_F \\
            & \leq
            \alpha\norm{x_t} + \alpha D + 4L_F.
        \end{split}
        \label{eq:x_t-to-v_t}
    \end{equation}
    We prove the bound on $\norm{x_t}^2$ using induction, and the bound on $\norm{v_t}^2$ directly follows from \eqref{eq:x_t-to-v_t}. The base case is true by the initialization $x_0\in\cC$, and thus $\norm{x_0}^2\leq D^2$. Assuming the claim holds for some $k\geq 0$, we now show the same is true for $k+1$. We consider the following two cases: $(a) \, \norm{x_k}^2< D^2$; $(b) \, D^2\leq\norm{x_k}^2\leq \gamma D^2 + \gamma(\gamma - 1)(D+4L_F/\alpha)^2$ for a given $\gamma>1$.
    
    For case $(a)$, we have that $\norm{v_k}< 2\alpha D + 4L_F$. Using $\alpha\eta_k \leq (\gamma - 1)/(\gamma + 1)$, we obtain
    \begin{align*}
        \norm{x_{k+1}}
        & \leq
        \norm{x_k} + \eta_k \norm{v_k} \\
        & <
        D + 2\alpha\eta_k\, D + 4 \eta_k L_F \\
        & \leq
        \frac{3\gamma-1}{\gamma+1}D + \frac{\gamma-1}{\gamma+1} \,\frac{4L_F}{\alpha}.
    \end{align*}
    Since $\gamma>1$, by basic arithmetic calculations, this implies that
    \begin{align*}
        \norm{x_{k+1}}^2
        & \leq
        \left(\frac{3\gamma - 1}{\gamma + 1}\right)^2 D^2 + \frac{2(3\gamma - 1)(\gamma - 1)}{(\gamma + 1)^2}\,\frac{4L_F D}{\alpha} + \left(\frac{\gamma - 1}{\gamma + 1}\right)^2\left(\frac{4L_F}{\alpha}\right)^2 \\
        & <
        \gamma^2 D^2 + 2\gamma(\gamma - 1)\frac{4L_F D}{\alpha} + \gamma(\gamma - 1)\left(\frac{4L_F}{\alpha}\right)^2 \\
        & =
        \gamma D^2 + \gamma(\gamma - 1)\left(D + \frac{4L_F}{\alpha}\right)^2.
    \end{align*}
    
    For case $(b)$, the constraint $g_{m+1}(x_k)\geq 0$ and thus enters the velocity polytope. By the definition of $V_\alpha(x_k)$, we know that $\alpha (\norm{x_k}^2 - D^2) + 2 x_k^\top v_k \leq 0$. This implies that
    \begin{align*}
        \norm{x_{k+1}}^2 - D^2
        & =
        \norm{x_k}^2 - D^2 + 2x_k^\top(x_{k+1} - x_k) + \norm{x_{k+1} - x_k}^2 \\
        & =
        \norm{x_k}^2 - D^2 + 2\eta_k\, x_k^\top v_k + \eta_k^2\norm{v_k}^2 \\
        & \leq
        (1 - \alpha\eta_k)(\norm{x_k}^2 - D^2) + \eta_k^2\norm{v_k}^2.
    \end{align*}
    Applying the inequality that $(a+b)^2\leq (1+1/\gamma)a^2 + (1+\gamma)b^2$ for any $a,b\in\bR$ and $\gamma>0$ to \eqref{eq:x_t-to-v_t}, we further have that for $\gamma > 1$ and $\alpha\eta_k\leq (\gamma - 1)/(\gamma + 1)$,
    \begin{align*}
        \norm{x_{k+1}}^2 - D^2
        & \leq
        (1 - \alpha\eta_k)(\norm{x_k}^2 - D^2) + (1+\gamma^{-1})\alpha^2\eta_k^2\norm{x_k}^2 + (1+\gamma)\alpha^2\eta_k^2\left(D + \frac{4L_F}{\alpha}\right)^2 \\
        & \leq
        (1 - \alpha\eta_k)(\norm{x_k}^2 - D^2) + (1 - \gamma^{-1})\alpha\eta_k\norm{x_k}^2 + (\gamma - 1)\alpha\eta_k\left(D + \frac{4L_F}{\alpha}\right)^2 \\
        & =
        \Big(1 - \alpha\eta_k + (1 - \gamma^{-1})\alpha\eta_k \Big)(\norm{x_k}^2 - D^2) + 
        (1 - \gamma^{-1})\alpha\eta_k D^2 + (\gamma - 1)\alpha\eta_k\left(D + \frac{4L_F}{\alpha}\right)^2 \\
        & =
        \left(1 - \frac{\alpha\eta_k}{\gamma}\right)(\norm{x_k}^2 - D^2) + (1 - \gamma^{-1})\alpha\eta_k D^2 + (\gamma - 1)\alpha\eta_k\left(D + \frac{4L_F}{\alpha}\right)^2.
    \end{align*}
    By the induction assumption, $\norm{x_k}^2\leq \gamma D^2 + \gamma(\gamma - 1)(D + 4L_F/\alpha)^2$, which concludes that $\norm{x_{k+1}}^2\leq \gamma D^2 + \gamma(\gamma - 1)(D + 4L_F/\alpha)^2$ as well.
    
    For both cases, we are able to prove that $\norm{x_{k+1}}^2\leq \gamma D^2 + \gamma(\gamma - 1)(D + 4L_F/\alpha)^2$. As a result, the bound on $\norm{x_t}^2$ holds for every $t=0,1,\cdots,T-1$. By \eqref{eq:x_t-to-v_t}, we also have that
    \begin{align*}
        \norm{v_t}^2
        & \leq
        \alpha^2(1 + \gamma^{-1})\norm{x_t}^2 + \alpha^2(1 + \gamma)\left(D + \frac{4L_F}{\alpha}\right)^2 \\
        & \leq
        (\gamma + 1)\alpha^2 D^2 + \gamma(\gamma + 1)\alpha^2\left(D + \frac{4L_F}{\alpha}\right)^2.
    \end{align*}
    This concludes the proof.
\end{proof}

\section{CGM with Simplex Constraints} \label{app:sp}

In this section, we provide details on how to derive the closed-form solutions of the quadratic program when applying CGM to problems with simplex constraints, which leads to the direct method stated in Algorithm \ref{algo:cgm-sp}. At each step, CGM solves the following quadratic program to decide the update direction:
\begin{align*}
    & \min_{v\in\bR^d} \; \frac{1}{2}\norm{v + F(x)}^2, \\
    & \text{s.t. } \quad \sum_{i=1}^d v_i = \alpha \left(1 - \sum_{i=1}^d x_i\right), \\
    & \phantom{\text{s.t. }} \quad v_i \geq -\alpha x_i, \;\forall i\in\{ j\in[d] \,|\, x_j \leq 0\},
\end{align*}
where we omit the subscript $t$ for simplicity and $x_i$ denotes the $i$-th coordinate of $x$. Let $p=v/\alpha + x$ and $q=x-F(x)/\alpha$. The above optimization problem is equivalent to
\begin{align*}
    & \min_{p\in\bR^d} \, \frac{1}{2}\norm{p - q}^2, \\
    & \text{s.t. } \quad \sum_{i=1}^d p_i = 1, \\
    & \phantom{\text{s.t. }} \quad p_i \geq 0, \;\forall i\in N,
\end{align*}
where the index set $N\subseteq[d]$ with size $n \in[0,d]$ denotes the set $\{ j\in[d] \,|\, x_j \leq 0\}$. Let the solution of the above problem be denoted by proj$_v(q, N)$. The actual update direction is then $v=\alpha(\text{proj}_v(x-F(x)/\alpha)-x)$, which gives the update in Algorithm \ref{algo:cgm-sp}.

Note that the major difference compared to a projection oracle onto the simplex is that we only restrict the coordinate in $N\subseteq[d]$ to be non-negative. We then derive Algorithm \ref{algo:proj-v} for solving the above quadratic program following the analysis in \citet{wang2013projection} for the projection algorithm on the simplex \citep{duchi2008efficient}. The Lagrangian function of the above problem is
\begin{equation*}
    \mathcal{L}(p, \lambda, \beta) = \frac{1}{2}\norm{p - q}^2 - \lambda \left(\sum_{i=1}^d p_i - 1\right) - \sum_{i\in N} \beta_i p_i,
\end{equation*}
and the corresponding KKT system has the following form:
\begin{align*}
    & \sum_{i=1}^d p_i = 1, \\
    & p_i - q_i - \lambda = 0, \,\forall i\notin N, \\
    & p_i - q_i - \lambda - \beta_i = 0, \, \forall i\in N, \\
    & p_i \geq 0, \forall i\in N, \\
    & \beta_i \geq 0, \forall i \in N, \\
    & p_i\beta_i = 0, \forall i \in N.
\end{align*}
The problem reduces to finding $\lambda$ and $\beta_i$. We observe that $\forall i\in N$, $\beta_i=0$ if $p_i>0$, and $q_i+\lambda = -\beta_i\leq 0$ if $p_i=0$. This means that $\forall i\in N$, $p_i=\max\{q_i+\lambda, 0\}$, and the sequence $\{p_i \,|\, i\in N\}$ shares the same ordering as $\{q_i \,|\, i\in N\}$. Since $\forall i\notin N, p_i=q_i+\lambda$, the problem left is to determine $\lambda$. Without loss of generality, we assume
\begin{align*}
    & q_1 \geq q_2 \geq \cdots \geq q_\rho \geq q_{\rho+1} \geq \cdots \geq q_n, \\
    & p_1 \geq p_2 \geq \cdots \geq p_\rho > p_{\rho+1} = \cdots = p_n. 
\end{align*}
Here, $\rho$ is the index such that $p_{\rho+1} = \cdots = p_n = 0$. Otherwise, a sorting of $\{q_i \,|\, i\in N\}$ can be applied. Let $s_{\bar N}=\sum_{i\notin N} q_i$. By the KKT system, we have that
\begin{align*}
    1 
    & = \sum_{i\in N} p_i + \sum_{i\notin N} p_i \\
    & = \sum_{i\in N, p_i> 0} (q_i + \lambda) + \sum_{i\notin N} (q_i + \lambda) \\
    & = (d - n + \rho) \lambda + s_{\bar N} + \sum_{i=1}^\rho q_i.
\end{align*}
This implies that the solution of $\lambda$ is
\begin{equation*}
    \lambda = \frac{1}{d-n+\rho} \left(1 - s_{\bar N} - \sum_{i=1}^\rho q_i\right).
\end{equation*}
The value of $\rho$ plays an important role. We then show that
\begin{equation*}
    \rho=
    \begin{cases}
        & 0, \quad \text{if } J=\emptyset, \\
        & \max J, \quad \text{otherwise},
    \end{cases}
\end{equation*}
where the set $J$ is defined to be
\begin{equation*}
    J=\left\{1\leq j \leq n \,\middle|\, q_j + \frac{1}{d-n+j}\left(1 - s_{\bar N} - \sum_{i=1}^j q_i\right) > 0\right\}.
\end{equation*}
Note that the algorithm includes the cases for $n=0$ and $n=d$. When $n=0$, we know $J=\emptyset$ and $\rho=0$. When $n=d$, it is easy to show that $J\neq\emptyset$ and the algorithm reduces to projection onto the simplex.

If  $J=\emptyset$, we have that
\begin{align*}
    q_1 + \frac{1}{d-n+1}(1 - s_{\bar N} - q_1)
    & =
    \frac{1}{d-n+1}(1 - s_{\bar N} + (d-n) q_1) \\
    & =
    \frac{d-n}{d-n+1}(q_1 + \lambda) + \frac{1}{d-n+1} \sum_{i=1}^\rho (q_i + \lambda) \\
    & \leq 0.
\end{align*}
This implies that $q_i+\lambda\leq 0, \forall i\in N$ and $\rho=0$. 

If $J\neq\emptyset$, we consider the following two cases. For $j\leq\rho$, we have that
\begin{align*}
    q_j + \frac{1}{d-n+j}\left(1 - s_{\bar N} - \sum_{i=1}^j q_i\right)
    & =
    \frac{1}{d-n+j}\left(1 - s_{\bar N} + (d - n + j)q_j - \sum_{i=1}^j q_i\right) \\
    & =
    \frac{1}{d-n+j}\left(1 - s_{\bar N} - \sum_{i=1}^\rho q_i + (d - n + j)q_j + \sum_{i=j+1}^\rho q_i\right) \\
    & =
    \frac{1}{d-n+j}\left((d - n + \rho)\lambda + (d - n + j)q_j + \sum_{i=j+1}^\rho q_i\right) \\
    & =
    \frac{1}{d-n+j}\left((d - n + j)(q_j + \lambda) + \sum_{i=j+1}^\rho (q_i+\lambda)\right) \\
    & > 0.
\end{align*}
For $j>\rho$, we have that
\begin{align*}
    q_j + \frac{1}{d-n+j}\left(1 - s_{\bar N} - \sum_{i=1}^j q_i\right)
    & =
    \frac{1}{d-n+j}\left(1 - s_{\bar N} - \sum_{i=1}^\rho q_i + (d - n + j)q_j - \sum_{i=\rho+1}^j q_i\right) \\
    & =
    \frac{1}{d-n+j}\left((d - n + \rho)\lambda + (d - n + j)q_j - \sum_{i=\rho+1}^j q_i\right) \\
    & =
    \frac{1}{d-n+j}\left((d - n + \rho)(q_j + \lambda) - \sum_{i=\rho+1}^j (q_i - q_j)\right) \\
    & \leq 0.
\end{align*}
This implies that $\rho=\max J$ and leads to the method described in Algorithm \ref{algo:proj-v}. Since the algorithm only requires to sort $n$ coordinates instead of $d$, it is provably more efficient compared to the projection oracle. The difference is particularly pronounced if the quadratic program is sparse.

\end{document}